\documentclass[12pt,twoside]{amsart}
\usepackage{amssymb,epsfig,amsmath,latexsym,amscd,amsthm}
\usepackage{tikz}
\usepackage{enumitem}
\usepackage{comment}

\nonstopmode

\numberwithin{equation}{section}
\font\tengothic=eufm10 scaled\magstep 1
\font\sevengothic=eufm7 scaled\magstep 1
\newfam\gothicfam
\textfont\gothicfam=\tengothic
\scriptfont\gothicfam=\sevengothic

\newcommand{\N}{\mathbb{N}}

\DeclareMathOperator{\pnt}{\raise 0.5mm \hbox{\large\bf.}}

\newtheorem{theorem}{Theorem}[section]
\newtheorem{lemma}[theorem]{Lemma}
\newtheorem{proposition}[theorem]{Proposition}
\newtheorem{corollary}[theorem]{Corollary}
\newtheorem{conjecture}[theorem]{Conjecture}

\theoremstyle{definition}
\newtheorem{definition}[theorem]{Definition} % \theoremstyle{remark}
\newtheorem{remark}[theorem]{Remark}
\newtheorem{example}[theorem]{Example}

\newtheorem{notation}[theorem]{Notation}

\author[G. Favacchio]{Giuseppe Favacchio}
\address{Dipartimento di Matematica e Informatica, Viale A. Doria, 6 - 95100 - Catania, Italy}
\email{favacchio@dmi.unict.it}

\author[J.~Migliore]{Juan Migliore}
\address{Department of Mathematics, University of Notre Dame, Notre Dame, IN 46556}
\email{migliore.1@nd.edu}

\title[multiprojective spaces and ACM points]{Multiprojective spaces and the arithmetically Cohen-Macaulay property}

\begin{document}
\thanks{Version: July 21, 2017}	
	
\begin{abstract} In this paper we study the arithmetically Cohen-Macaulay (ACM) property for sets of points in multiprojective spaces. Most of what is known is for $\mathbb P^1\times \mathbb P^1$  and, more recently, in $(\mathbb P^1)^r.$ In $\mathbb P^1\times \mathbb P^1$ the so called inclusion property characterizes the ACM property. We extend the definition in any multiprojective space and we prove that the inclusion property implies the ACM property in $\mathbb P^m\times \mathbb P^n$. In such an ambient space it is equivalent to the so-called $(\star)$-property.
Moreover, we start an investigation of the ACM property in 	$\mathbb P^1\times \mathbb P^n.$ We give a new construction that highlights how different the behavior of the ACM property is in this setting.

\end{abstract}

\maketitle

%%%%%%%%%%%%%%%%%%%%%%%%%%%%%%%%%%%%%%%%%%%%%%%%%%%%%%%%%%%

\section{Introduction}

%A motivating problem in algebraic geometry and commutative algebra concerns multiprojective spaces.
Let   $X\subseteq \mathbb P^{a_1} \times \dots \times \mathbb P^{a_n}$ be a finite collection of points. It is interesting to describe the homological invariants of the coordinate ring of $X$. 
In particular, it is a subject of research to understand when $X$ is arithmetically Cohen-Macaulay (ACM), i.e. when the coordinate ring is a Cohen-Macaulay ring. Since it is no longer the case (as it is in projective space) that a finite set of points is automatically ACM, the determination of whether or not the ACM property holds for a finite set draws on a combination of geometric, combinatoric, algebraic and numerical considerations. The Cohen-Macaulay question here is closely related to the Cohen-Macaulay question for unions of linear varieties in projective space, but it is a more manageable version of the problem than the case of arbitrary unions since only certain such unions correspond to finite sets in multiprojective spaces. Indeed, we will frequently use this connection.

A characterization of finite sets of points with the ACM property is only known in $\mathbb P^1 \times \mathbb P^1$ (see \cite{GV-book} for an exhaustive discussion of the topic) and, more recently, in $\mathbb P^{1} \times \mathbb P^{1}\times \dots \times\mathbb P^{1} = (\mathbb P^1)^n$ (see \cite{FGM2017}). The purpose of this paper is to investigate the new subtleties that arise in studying the ACM property in more general multiprojective spaces.

More precisely, given $X$ a finite collection of points in $\mathbb P^{a_1} \times \mathbb P^{a_2}$, one can define the so-called $(\star)$-property (see \cite{GV-book} Definition 3.19 or page \pageref{star prop} of this paper for the definition). It is known that a collection of points $X$ in $\mathbb P^1 \times \mathbb P^{1}$  is ACM if and only if it satisfies the $(\star)$-property (cf. for instance \cite{GV-book} Theorem 4.11). In \cite{FGM2017} a characterization of the ACM property for finite sets was obtained for $(\mathbb P^1)^n$, in terms of what was called the $(\star_n)$-property (see \cite{FGM2017} Definition 3.6 and Theorem 3.13). What is significant about both of these results is that the ACM question is determined by the existence or not of certain kinds of subconfigurations.

Here we prove that in $\mathbb P^{a_1} \times \mathbb P^{a_2}$, if $X$ is a set of points satisfying the $(\star)$-property then $X$ is ACM (Theorem \ref{star implies ACM}). It was already known that the converse does not hold (see for instance \cite{GuVT2008b}), but we give simpler examples. Part of the purpose of section \ref{sec:P1xPn} is to get a better understanding of the fact that the converse is not true (see Remark \ref{subconfigurations}). 

In \cite{FGM2017} the authors introduce, for a set of points in $(\mathbb P^1)^n$, the {\em inclusion property}. 
In section~\ref{sec:inclusion} we generalize the inclusion property to an arbitrary multiprojective space $\mathbb P^{a_1} \times \dots \times \mathbb P^{a_n}$ (see Definition \ref{inclusion prop}). We show that when $n=2$ it is still true that the inclusion property is equivalent to the $(\star)$-property (Lemma \ref{star iff inclusion}), and hence it implies the ACM property for $X$ (Corollary \ref{inclusion implies ACM}). We are also able to show that if any of the $a_i$ is equal to~1 then again the inclusion property implies ACM (Proposition \ref{inclusion-ACM}). We conjecture that this implication holds in general.
%{\bf (Unless we can prove it...)}

To investigate ACM sets of points in $\mathbb P^{a_1} \times \mathbb P^{a_2} \times \dots \times \mathbb P^{a_n}$  one can use extensions, to the multihomogeneous setting, of some standard tools in the homogeneous setting. These include {\em hyperplane sections, basic double G-linkage, liaison addition, and liaison}. This approach was already used in \cite{FGM2017}.
For example, some of our results are for $\mathbb P^1 \times \mathbb P^{a_2} \times \dots \times \mathbb P^{a_n}$. Omitting details here, we observe that the distinguishing feature of this case for us is that in $\mathbb P^1$ a point is also a hyperplane, and this allows us to use hyperplane sections and related constructions in our study. 

This is crucial for instance in the proof of Proposition \ref{inclusion-ACM}, that the inclusion property implies the ACM property in $\mathbb P^1 \times \mathbb P^{a_2} \times \dots \times \mathbb P^{a_n}$. Thus it was surprising to us when we obtained Corollary \ref{inclusion implies ACM}, that the inclusion property implies the ACM property in $\mathbb P^{a_1} \times \mathbb P^{a_2}$, which avoids hyperplane sections but reaches the same conclusion.

In Section \ref{sec:P1xPn} we explore the ACM property for collection of points in $\mathbb P^1\times \mathbb P^n$. Examples \ref{ex:4 pt ACM} and \ref{ex:6pt Not ACM} underline a crucial difference with the $(\mathbb P^1)^n$ case. 
A set $X$ of reduced points of $(\mathbb P^1)^n$ has the ACM property if and only if it does not contain certain  sub-configurations (see \cite{FGM2017} Theorem 3.13). A similar characterization is not possible in $\mathbb P^1\times\mathbb P^n$. Instead, we give a construction of a set of points which, as one continues to add points following the same prescribed rules, fluctuates between being ACM and not being ACM in a predictable way.
Section \ref{sec:P1xPn}  is devoted to a careful study of this construction and what it tells us about the ACM property. %{\bf (Should we say more about section \ref{sec:P1xPn}?)}

%%%%%%%%%%%%%%%%%%%%%%%%%%%%%%%%%%%%%%%%%%%%%%%%%%%%%%%%%%%

\section{Preliminaries}\label{sec:preliminaries}

We work over a field of characteristic zero.
Set $S:=k[\mathbb P^n]$. 
Recall that for a finite set of points $Z\subset \mathbb P^n$ the Hilbert function of $Z$ is defined as the numerical function $H_Z: \mathbb N \to \mathbb N$ such that
\[
H_Z(i)=\dim_k(S/I_Z)_i= \dim_k S_i-\dim_k(I_Z)_i.
\]
Since $H_Z(t)=\# Z$ for $t$ large enough, the first difference of the Hilbert function $\Delta H_Z(i):=H_Z(i)-H_Z(i-1)$ is eventually  zero. The {\em $h$-vector} of Z is 
\[
h_Z= h = (1,h_1,\ldots, h_t )
\]
where $h_i  = \Delta H_Z(i)$ and $t$ is the last index such that $\Delta H_Z(i)  > 0$. 

A finite set of points in $\mathbb P^n$ is said to have {\em generic Hilbert function} if $H_Z(i):=\min\left\{{i+n \choose n}, \# Z\right\}$, i.e. $h_Z= (1,{n \choose n-1}, {n+1 \choose n-1}, \cdots, \binom{n+t-2}{n-1}, h_t  )$
where $0<h_t\le {n+t-1 \choose n-1}$.

\begin{definition}
	For $V = \mathbb P^{a_1} \times \dots \times \mathbb P^{a_n}$ we define 
	\[
	\pi_i : V \rightarrow \mathbb P^{a_1} \times \dots \times \widehat{\mathbb P^{a_i}} \times \dots \times \mathbb P^{a_n}
	\] 
	to be the projection omitting the $i$-th component and 
	\[
	\eta_i : V \rightarrow \mathbb P^{a_i}
	\]
	to be the projection to the $i$-th component.
Note that if $V:=\mathbb P^{a_1} \times \mathbb P^{a_2}$ then $\pi_1=\eta_2.$
\end{definition}

Let $\underline{e}_1,\dots, \underline{e}_n$ be the standard basis of $\mathbb N^n$.
Let $x_{i,j}$, with $1 \leq i \leq n$ and $0 \leq j \leq a_i$ for all $i,j$, be the variables for the different $\mathbb P^{a_i}$. Let 
\[
R = k[x_{1,0}, \dots,x_{1,a_1},\dots, x_{n,0}, \dots, x_{n,a_n}],
\]
where the degree of $x_{i,j}$ is $\underline{e}_i$. 

A subscheme $X$ of $V$ is defined by a saturated ideal, $I_X$, generated by a system of multihomogeneous polynomials in $R$ in the obvious way.  We say that $X$ is {\em arithmetically Cohen-Macaulay (ACM)} if $R/I_X$ is a Cohen-Macaulay ring.

Let $N = a_1 + \dots + a_n + n$.
Given a subscheme $X$ of $V$ together with its homogeneous ideal $I_X$, we can also consider the subscheme $\bar X$ of $\mathbb P^{N-1}$ defined by $I_X$. Notice that if $X$ is a zero-dimensional subscheme of $V$, $I_X$  almost never defines a zero-dimensional subscheme of $\mathbb P^{N-1}$.% For example, if $n=2, a_1 = a_2 = 1$, then a finite subset, $X$, of $\mathbb P^1 \times \mathbb P^1$ corresponds to a finite union of lines, $\bar X$, in $\mathbb P^3$ (of a certain type). The subscheme $X \subset V$ is ACM if and only if the subscheme $\bar X \subset \mathbb P^{N-1}$ is ACM.

The following definition also includes facts that can be found in the literature. It is a special case of so-called {\em Basic Double Linkage}. See for instance \cite{MN}~Lemma 3.4 and Corollary 3.5, \cite{GV-book}~Theorem~4.9 and \cite{FGM2017}~Proposition 2.3.

\begin{definition} \label{BDL}
Let $V_1 \subseteq V_2 \subseteq \dots \subseteq V_r \subset \mathbb P^n$ be ACM of the same
dimension  $\geq 1$. Let $H_1, \dots, H_r$ be hypersurfaces, defined by forms $F_1,\dots, F_r$, such that
for each $i$, $H_i$ contains no component of $V_j$ for any $j \leq i$. Let $W_0 \subset V_1$ be a codimension 1 ACM
subscheme, and for each $i \geq 1$ let $W_i$ be the ACM scheme defined by the corresponding
hypersurface sections: $I_{W_i} = I_{V_i} + (F_i)$. Let $Z$ be the sum of the $W_i$, viewed as divisors
on $V_r$. Then

\begin{itemize}
\item[(i)] $Z$ is ACM.

\item[(ii)] As ideals we have
\[
I_Z = I_{V_r} + F_r I_{V_{r-1}} + F_r F_{r-1} I_{V_{r-2}} + \dots + F_r F_{r-1} \dots F_2 I_{V_1} + F_r F_{r-1} \dots F_1 I_{W_0}.
\]
\end{itemize}
\end{definition}

%%%%%%%%%%%%%%%%%%%%%%%%%%%%%%%%%%%%%%%%%%%%%%%%%%%%%%%%%%%

\section{The inclusion property and the $(\star)$-property}\label{sec:inclusion}

The next definition introduces a partition for finite subsets of $\mathbb P^{a_1}\times \mathbb P^{a_2}\times \cdots\times \mathbb P^{a_n}$. Without loss of generality we focus on the projection to the first component, but the definition could just as well be made for any of the projections. See also Theorem 3.21 of~\cite{GV-book}.

%\begin{definition} \label{old inclusion prop}
%	Let $X \subset \mathbb P^1\times \mathbb P^{a_2}\times \cdots\times \mathbb P^{a_n}$ be a finite, reduced subscheme.   Let $\{[k_1,\ell_1],\dots,[k_t, \ell_t] \} = \eta_1(X)$.
%	For $1 \leq j \leq t$, let $\mathbb H_{j}$ be the hyperplane defined by $\ell_j x_{1,0} - k_j x_{1,1}$ and let  $X_j = X \cap \mathbb H_{j}$. We call the $X_j$ the {\em level sets} of $X$. We say that $X$ has the {\em inclusion property with respect to $\pi_1$} if  the subsets $\pi_1 (X_j)$ of $\mathbb P^{a_2}\times \cdots\times \mathbb P^{a_n}$, for $1 \leq j \leq t$, admit a total ordering by inclusion and are all ACM. 
%\end{definition}

\begin{definition} \label{inclusion prop}
Let $X \subset \mathbb P^{a_1}\times \mathbb P^{a_2}\times \cdots\times \mathbb P^{a_n}$ be a finite, reduced subscheme. Let $\eta_1(X) = \{ P_1,\dots, P_t \} \subset \mathbb P^{a_1}$.   
For each $P_j \in \eta_1(X)$ let $X_j = \eta_1^{-1}(P_j) \cap X$.
We call the $X_j$ the {\em level sets} of $X$ with respect to $\eta_1$. We say that $X$ has the {\em inclusion property with respect to $\pi_1$} if  the collection of subsets $\{ \pi_1(X_1), \dots, \pi_1(X_t)  \} $ of $\mathbb P^{a_2}\times \cdots\times \mathbb P^{a_n}$, for $1 \leq j \leq t$, is totally ordered by inclusion and they are all ACM. 
\end{definition}

The next proposition gives a relation between the inclusion property and the  ACM property for finite sets of points when $a_1 = 1$.

\begin{proposition} \label{inclusion-ACM}
	Let $X \subset \mathbb P^1\times \mathbb P^{a_2}\times \cdots\times \mathbb P^{a_n}$ be a finite set. Let $X_1,\dots,X_t $ be the level sets with respect to $\eta_1$, let $Y_i = \pi_1(X_i) \subset \mathbb P^{a_2} \times \dots \times \mathbb P^{a_n}$ for each $i$, and let $L_1, \dots, L_t \in k[x_0,x_1]$ be the linear forms defining the points $\{P_1,\dots,P_t\} \subset \mathbb P^1$ as in Definition \ref{inclusion prop}.
Assume that $X$ has the inclusion property with respect to $\pi_1$. In particular, each $Y_i$ is ACM and we can assume that $Y_1 \supset Y_2 \supset \dots \supset Y_t$.
Then $X$ is ACM. Furthermore, 
\[
I_X =  I_{Y_1} + L_1 I_{Y_2} + L_1 L_2 I_{Y_3} + \dots + L_1 L_2 \dots L_{t-1} I_{Y_t} + (L_1 L_2 \dots L_t).
\]
\end{proposition}
\begin{proof}
It follows from Definition \ref{BDL}, viewing this in $\mathbb P^{a_2 + \dots + a_n + n}$ -- see the proof of Proposition 2.6 in \cite{FGM2017}. Note that here we use $I_{W_0} = (I_{Y_t},L_t)$. What is important about $\mathbb P^1$ is that in that case the level sets are hyperplane sections of ACM varieties, because points in $\mathbb P^1$ are hyperplanes.
\end{proof}

If the ambient space only consists of a product of two projective spaces, $\mathbb P^{a_1}\times \mathbb P^{a_2},$ we now define the so-called {\em $(\star)$-property} (or \textit{star} property), following \cite{GV-book} Definition 3.19. 
%It is easy to check that 
%if $a_1=1$, 
%the inclusion property and the $(\star)$-property agree. {\bf Originally this was stated only for $a_1=1$ but I don't see why this assumption is needed.  It's not at all obvious to me, though, and I think it needs to be a lemma.}

\begin{definition}\label{star prop}
	A finite set $X\subset \mathbb P^{a_1}\times \mathbb P^{a_2}$ has the $(\star)$-property if and only if for any $(P_1,Q_1), (P_2,Q_2)\in X \subseteq \mathbb P^{a_1}\times \mathbb P^{a_2}$ then also either $(P_1,Q_2)$ or $ (P_2,Q_1)\in X$. 
\end{definition}

The inclusion property and the $(\star)$-property agree in $\mathbb P^{a_1}\times \mathbb P^{a_2}$. This fact is known; it was shown using a different notation for sets of points in $\mathbb P^1 \times \mathbb P^1 $ (see for instance Theorem 3.21. in \cite{GV-book}). For completeness we include a proof.

\begin{lemma} \label{star iff inclusion}
If $X \subset \mathbb P^{a_1}\times \mathbb P^{a_2}$ is a finite set, then $X$ satisfies the inclusion property if and only if it satisfies the $(\star)$-property.
\end{lemma}

\begin{proof}
Assume that $X$ satisfies the inclusion property. Notice that in this case the ACM condition for the inclusion property is trivial. Then we can label the elements of $\eta_1(X)$ so that there is a sequence of points $P_1,P_2,\dots,P_s \in \mathbb P^{a_1}$ with
\[
X  =  X_1 \cup \dots \cup X_s 
\]
being the level set decomposition, and
\[
\begin{array}{rcl}
X_1 & = & \{ (P_1,Q_{1,1}),\dots, (P_1,Q_{1,n_1}) \} \\ 
& \vdots \\
X_s & = & \{ (P_s,Q_{s,1}), \dots, (P_s,Q_{s,n_s}) \}
\end{array}
\]
and furthermore
\begin{equation} \label{inclusion in pf}
\mathbb P^{a_2} \supset \{ Q_{1,1},\dots,Q_{1,n_1} \} \supseteq \dots \supseteq \{Q_{s,1},\dots,Q_{s,n_s} \}.
\end{equation}
Then it is clear that $X$ satisfies the $(\star)$-property. 

Conversely, assume that $X$ satisfies the $(\star)$-property and suppose that it does not have the inclusion property. Then $X$ is decomposed into level sets as above, but the inclusions (\ref{inclusion in pf}) do not all hold. Without loss of generality, suppose that $A_1 = \{ Q_{1,1},\dots,Q_{1,n_1} \}$ and $A_2 = \{ Q_{2,1},\dots,Q_{2,n_2} \}$ are incomparable with respect to inclusion. Specifically, suppose $Q_{1,1} \notin A_2$ and $Q_{2,1} \notin A_1$. Then $(P_1,Q_{1,1})$ and $(P_2,Q_{2,1})$ violate the $(\star)$-property.
\end{proof}

A set of points with the $(\star)$-property (equivalently the inclusion property) can also be organized  as ``rectangles" in the following way. For convenience we now start indexing with 0 rather than 1. If $X\subset \mathbb P^{a_1}\times \mathbb P^{a_2}$ has the $(\star)$-property then, after renaming, we can always assume that there exists a set 
\[
U(X):=\{(i_1,j_1), (i_2,j_2),\ldots (i_t,j_t)\}\subseteq \N^2
\]
 where $i_1> \cdots > i_t$ and $j_1< \cdots < j_t$, such that 
\[
X=\{P_i\times Q_j\ | \ (0,0)\le  (i,j)\le (i_k,j_k)\ \text{for some}\ (i_k,j_k)\in U(X) \}.
\]

\bigskip
\bigskip

\begin{figure}\label{figure1}
\begin{picture}(160,160)(10,10)
\thinlines
\put (10,10){\vector(0,1){170}}
\put (10,10){\vector(1,0){170}}
\put (10,160){\line(1,0){50}}
\put (60,10){\line(0,1){150}}
\put (160,10){\line(0,1){40}}
\put (10,50){\line(1,0){150}}
\put (10,140){\line(1,0){70}}
\put (80,10){\line(0,1){130}}
\put (120,10) {\line(0,1){105}}
\put (10,115) {\line(1,0){110}}
\put (135,10) {\line(0,1){80}}
\put (10,90) {\line(1,0){125}}
%\thicklines

\put (63,165){$(i_5,j_5)$}
\put (84,143){$(i_4,j_4)$}
\put (122,120){$(i_3,j_3)$}
\put (138,94){$(i_2,j_2)$}
\put (162,56){$(i_1,j_1)$}

\put (57,157){$\bullet$}
\put (77,137){$\bullet$}
\put (117,112){$\bullet$}
\put (132,87){$\bullet$}
\put (157,47){$\bullet$}

\end{picture}
\caption{Example of $U(X)$ configuration.}
\end{figure}
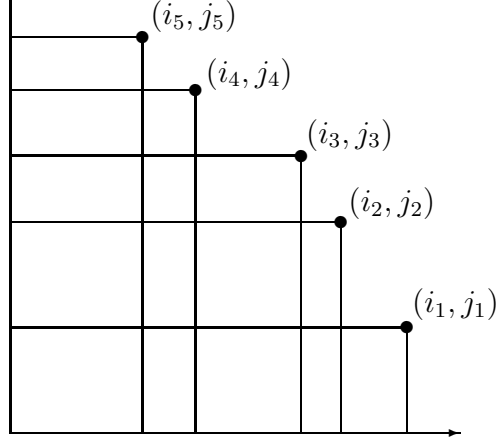

\bigskip  \bigskip

\noindent Moreover, in this case, we set 
\[
V_k:=\{P_i\in \pi_2(X) \ |\  i\le i_k\}\subseteq \mathbb P^{a_1}\ \text{and}\ Z_k:=\{Q_j \in \pi_1(X) \ |\ j\le j_k\}\subseteq \mathbb P^{a_2}
\]
 for $h=1,\ldots, t,$ where $t$ is, as above, the number of elements in $U(X)$.
 
 We first consider the case of just one ``rectangle."

\begin{lemma}\label{rectangular case}
Let $X\subset \mathbb P^{a_1} \times \mathbb P^{a_2}$ be a finite set of points. Assume that $X$ has the $(\star)$-property and $U(X)=\{(i_1,j_1)\}$.
Then 
\begin{enumerate}
	\item $I_X=I_{V_{1}}+I_{Z_{1}}$;
	\item $X$ is ACM.
\end{enumerate}
\begin{proof}
%	Let denote for short by $x_0,\ldots, x_{a_1}$ and $ y_0,\ldots, y_{a_2}$ the variable of degree $(1,0)$ and $(0,1)$ respectively. 	
It is trivial to check that %$I_X\supseteq I_{V_{1}}+I_{Z_{1}}.$ On the other hand, 
\[
I_X=\bigcap_{(r,s)\le (i_1,j_1) } (I_{P_r}+I_{Q_s}) = \bigcap_{r=0}^{i_1}  (I_{P_r}+ I_{Z_1}) = I_{V_{1}}+I_{Z_{1}}.
\]
Moreover, note that $V_1$ and $Z_1$ are both ACM, and  
\[
R/I_X \cong k[x_0,\dots,x_{a_1}]/I_{V_1} \otimes_k k[y_0,\dots,y_{a_2}]/I_{Z_1}
\]
so $X$ is ACM.
\end{proof}
\end{lemma}

\begin{theorem}\label{thm:decomposition}
Let $X\subset \mathbb P^{a_1} \times \mathbb P^{a_2}$ be a finite set of points. Assume that $X$ has the $(\star)$-property and $U(X)=\{(i_1,j_1), (i_2,j_2),\ldots (i_t,j_t)\}$.
Then 
\[
I_X=I_{V_{1}}+I_{V_{{2}}}I_{Z_{1}}+\cdots+I_{V_{{t}}}I_{Z_{{t-1}}}+I_{Z_{t}}.
\]
\end{theorem}

\begin{proof}
First we claim that
\[
I_X\supseteq I_{V_{1}}+I_{V_{{2}}}I_{Z_{1}}+\cdots+I_{V_{{t}}}I_{Z_{{t-1}}}+I_{Z_{t}}.
\]
We show that each summand is contained in $I_X$ (see the Figure \ref{figure1}). First, clearly we have $I_{V_{1}},I_{Z_{t}}\subseteq I_X$. Moreover, let $F=G\cdot G'\in I_{V_{{r}}}I_{Z_{r-1}}$ and $P_i\times Q_j\in X.$ If $i\le i_r$ then $P_i\in V_r$ and $G$ vanishes at $P_i$. Else we have $i> i_r$ and $(i,j)\le (i_s,j_s)$ for some $(i_s,j_s)\in U(X)$ where $i_s\ge i> i_r$. Thus, we get $j\le j_s< j_r$ i.e. $j\le j_s\le j_{r-1}$ and $Q_j\in Z_{r-1}.$   

To prove the other inclusion we proceed by induction on $|U(X)|.$ 
The base of the induction follows from Lemma \ref{rectangular case}. Assume now $|U(X)|>1$. We introduce the following partition on $X$:
\[
X=Y_0 \cup Y_1
\]
where
\[
Y_0:=\{P_i\times Q_j \ | \  P_i\in  V_{1}\setminus V_{2}, Q_j \in Z_{1} \} 
\hbox{ and } 
Y_1 =X\setminus Y_0.
\]
Note that $Y_0$ and $Y_1$ have the $(\star)$-property and $U(Y_1)=U(X)\setminus \{(i_1,j_1)\}.$
Set 
\[
J:= I_{V_{1}}+I_{V_{{2}}}I_{Z_{1}}+\cdots+I_{V_{{t}}}I_{Z_{{t-1}}}+I_{Z_{t}}.
\]
We want to show that $I_X \subseteq J$.

We first claim that
\begin{equation} \label{identification}
 (I_{V_{1} \setminus V_{2}}+ I_{Z_1}) \cap I_{V_{2}} = I_{V_1} + I_{V_2} I_{Z_1}.
\end{equation}
Indeed, it is clear that both ideals define the same scheme, and from Lemma \ref{rectangular case} we see that the left-hand side is saturated. We just have to prove that $I_{V_1} + I_{V_2} I_{Z_1}$ is also saturated. Consider the exact sequence
\begin{equation} \label{SES1}
0  \rightarrow I_{V_1} \cap I_{V_2} I_{Z_1} \rightarrow I_{V_1} \oplus I_{V_2} I_{Z_1} \rightarrow I_{V_1} + I_{V_2} I_{Z_1} \rightarrow 0.
\end{equation}
By Lemma \ref{rectangular case},  $I_{V_2} I_{Z_1}$ is a saturated ideal. Hence $I_{V_1} \cap I_{V_2} I_{Z_1} = I_{V_1} \cap I_{Z_1}$. From the exact sequence
\[
0 \rightarrow I_{V_1} \cap I_{Z_1} \rightarrow I_{V_1} \oplus I_{Z_1} \rightarrow I_{V_1} + I_{Z_1} \rightarrow 0
\]
and Lemma \ref{rectangular case}, by viewing the ideals in $R$ and the schemes in $\mathbb P^{a_1+a_2+1}$, sheafifying and taking cohomology we see that $H^1_*(\mathcal I_{V_1} \cap I_{Z_1}) = 0$ (even though the scheme is not unmixed if $a_1 \neq a_2$) since $V_1$ and $Z_1$ are ACM. Putting it together with (\ref{SES1}), cohomology gives that 
$I_{V_1} + I_{V_2} I_{Z_1}$ is saturated, as desired, proving our claim of (\ref{identification}).

By induction and by Lemma \ref{rectangular case} we have
\[
I_X=(I_{V_{1}\setminus V_{2}}+I_{Z_{1}})\cap (I_{V_{{2}}}+ I_{V_{{3}}}I_{Z_{{2}}}+\cdots+I_{V_{{t}}}I_{Z_{{t-1}}}+I_{Z_{t}}).
\]
Let $F \in I_X$.  In particular we have $F \in I_{V_{{2}}}+ I_{V_{{3}}}I_{Z_{{2}}}+\cdots+I_{V_{{t}}}I_{Z_{{t-1}}}+I_{Z_{t}}$, so 
$F = F' + H$ where $F' \in I_{V_2}$ and $H \in I_{V_{{3}}}I_{Z_{{2}}}+\cdots+I_{V_{{t}}}I_{Z_{{t-1}}}+I_{Z_{t}} \subseteq J \subseteq I_X$. Since both $F$ and $H$ are in $I_X$, it follows that $F' \in I_X$. Hence using (\ref{identification}) we obtain  
\[
F' \in I_X \cap I_{V_2} = (I_{V_{1} \setminus V_{2}}+ I_{Z_1}) \cap I_{V_{2}} = I_{V_1} + I_{V_2} I_{Z_1} \subset J.
\]
%{\bf You originally put $\subseteq$ instead of the first equality above. Do you agree that it is $=$ or do you want to put back $\subseteq$?} 
Since $H \in J$ and $F' \in J$, we have $F \in J$ and we are finished.
\end{proof}

\begin{theorem}\label{star implies ACM}
If $X\subset \mathbb P^{a_1} \times \mathbb P^{a_2}$ is a finite set of points with the $(\star)$-property, then $X$ is ACM.
\end{theorem}

\begin{proof}
We proceed by induction on $|U(X)|.$ If $U(X)=\{(i_1,j_1)\}$ the statement follows by Lemma \ref{rectangular case}.
Now let  $|U(X)|>1$. We can decompose $X$ as follows:
\[
X=\big((V_{1}\setminus V_{2})\times Z_{1}\big) \cup Y
\]
with $U(Y)=U(X)\setminus \{(i_1,j_1)\}.$
From this partition for $X$ we obtain the following short exact sequence:
\begin{equation} \label{SES2}
0 \to I_X \to I_{(V_{1}\setminus V_{2}) \times Z_{1}} \oplus I_{Y} \to I_{(V_{1}\setminus V_{2}) \times Z_{1}} + I_{Y}\to 0
\end{equation}
where by induction $(V_{1}\setminus V_{2})\times Z_{1}$ and $Y$ are both ACM. As subschemes of $\mathbb P^{a_1 + a_2 + 1}$ they are reduced unions of lines, and so in particular the first cohomology of their ideal sheaves vanish (see for instance \cite{J book} Lemma 1.2.3).

Moreover, by Theorem \ref{thm:decomposition} and induction we get 
\[
I_{(V_{1}\setminus V_{2})\times Z_{1}}+ I_{Y} = 
I_{V_{1} \setminus V_{2} } + I_{Z_{1} }+ I_{V_{{2}}} + I_{V_{{3}}} I_{Z_{2}} + \cdots  + I_{V_{{t}}}I_{Z_{{t-1}}}+I_{Z_{t}}.
\]
Since $Z_1\subseteq Z_j$ for any $j\ge 1$, the right-hand side  simplifies and we obtain
\[
I_{(V_{1}\setminus V_{2})\times Z_{1}}+ I_{Y}  = I_{V_{1}\setminus V_{2} }+ I_{V_{{2}}}+ I_{Z_{1}}.
\]
Now $k[\mathbb P^{a_1}]/(I_{(V_1 \setminus V_2)} + I_{V_2})$ is artinian (in particular Cohen-Macaulay) and $k[\mathbb P^{a_2}]/I_{Z_1}$ is Cohen-Macaulay of depth 1. Thus 
\[
R/ (I_{V_1 \setminus V_2 } + I_{V_2} + I_{Z_1}) \cong k[\mathbb P^{a_1}]/(I_{(V_1 \setminus V_2)} + I_{V_2}) \otimes_k k[\mathbb P^{a_2}]/I_{Z_1}
\]
is Cohen-Macaulay, defining a zero-dimensional scheme in $\mathbb P^{a_1 + a_2 + 1}$. In particular, 
\[
I_{(V_{1}\setminus V_{2})\times Z_{1}}+ I_{Y}  
\]
is a saturated ideal. Then sheafifying (\ref{SES2}) and taking cohomology, we see that $H^1_*(\mathcal I_X) = 0$, i.e. $X$ is ACM (see \cite{J book}  Lemma 1.2.3).
\end{proof}

%Note that $I_{Z_{1}}$ is minimally generated by bihomogeneous forms of degree $(0, *)$ and if we look at it as an  ideal of $k[\mathbb P^{a_2}]$ we have $\depth k[\mathbb P^{a_2}]/I_{Z_{1}}=1$. Furthermore  $I_{V_{1}\setminus V_{2}}  + I_{V_2}$ is generated by forms of degree $(*,0)$ i.e. a minimal set of generators for this ideals involves a different set of variables with respect to $I_{Z_{1}}$. Thus 
%	$\depth R/(I_{V_{1}\setminus V_{2} }+I_{V_{{2}}}+ I_{Z_{1}})\ge 1$ \textbf{(is it true?)} so  
%	$$\pd R/I_{V_{1}\setminus V_{2} }+ I_{V_{{2}}}+ I_{Z_{1}}\le a_1+a_2+1.$$
%	By a mapping cone argument we get $X$ ACM. 

\begin{corollary} \label{inclusion implies ACM}
Let $X \subset \mathbb P^{a_1} \times \mathbb P^{a_2}$ be a set of points with the inclusion property. Then $X$ is ACM.
\end{corollary}

\begin{proof}
It follows from Lemma \ref{star iff inclusion} and Theorem \ref{star implies ACM}.
\end{proof}

\begin{remark}
We have defined the inclusion property for a product of any number of
projective spaces. We know from Proposition \ref{inclusion-ACM} that when one of the projective spaces
is $\mathbb P^1$ then the inclusion property implies ACM. Furthermore, we have just seen in Corollary \ref{inclusion implies ACM} that if
we have a product of only two projective spaces then again the inclusion property implies ACM, whether
or not one of them is  $\mathbb P^1$. This motivates
the following conjecture.
\end{remark}

\begin{conjecture} \label{inclusion-ACM-conj}
 Let $X \subset \mathbb P^{a_1} \times \dots \times \mathbb P^{a_n}$ be a set of points with the inclusion property. Then $X$ is ACM.
 \end{conjecture}

%\begin{question} {\bf (For Giuseppe)} 
%We have defined the inclusion property for a product of any number of projective spaces. We know that when one of the projective spaces is $\mathbb P^1$ then the inclusion property implies ACM. We also know that if none of them is $\mathbb P^1$ but there are only two projective spaces, then the result again holds. Can we show that the inclusion property always implies ACM?
%
%{\bf (For Juan)} From Lemma \ref{rectangular case} to Theorem \ref{star implies ACM}, we work in $\mathbb P^{a_1}\times\mathbb P^{a_2}$. Do you see any obstacles in change the ambient space with  $\mathbb P^{a_1}\times\mathbb  U$ where $\mathbb U:=\mathbb P^{a_2}\times \cdots \times\mathbb P^{a_n}$?  I mean, keeping the same proofs.  
%
%{\bf (For Giuseppe)} I don't know, but my feeling is that it's not so simple.}
%
%\end{question}

\begin{remark}
The results of this section, especially Theorem \ref{thm:decomposition}, Theorem \ref{star implies ACM} and Corollary \ref{inclusion implies ACM}, can be viewed as an extension of the notion of basic double G-linkage, a multihomogeneous version of which was used  in \cite{FGM2017}. 
\end{remark}

%%%%%%%%%%%%%%%%%%%%%%%%%%%%%%%%%%%%%%%%%%%%%%%%%%%%%%%%%%%%

\section{ACM sets of points in $\mathbb P^1\times\mathbb P^n$.}\label{sec:P1xPn}

This section is devoted to a further examination of ACM sets of points in a multiprojective space $\mathbb P^1\times\mathbb P^n$. We denote the coordinate ring of $\mathbb P^1\times\mathbb P^n$ by
  \[
 R = k[x_{0},x_{1}, y_{0}, \dots, y_{n}],
 \]
 where  $\deg (x_{i})=(1,0)$  and $\deg (y_{j})=(0,1)$.

\begin{remark}
We have seen in Lemma \ref{star iff inclusion} that for  $\mathbb P^{a_1} \times \mathbb P^{a_2}$, the $(\star)$-property is equivalent to the inclusion property, which is in fact a more generally defined notion. We have seen in Proposition \ref{inclusion-ACM} and Corollary \ref{inclusion implies ACM} that in $\mathbb P^1 \times \mathbb P^{a_2} \dots \times \mathbb P^{a_n}$ and $\mathbb P^{a_1} \times \mathbb P^{a_2}$, respectively, the inclusion property implies the ACM property. What about the converse?

The ($\star$)-property characterizes the ACM property in $\mathbb P^1 \times \mathbb P^1$; see for instance Theorem 4.11 in \cite{GV-book}.  Thus in $\mathbb P^1 \times \mathbb P^1$ the converse holds. However, 
Example 2.12 of \cite{FGM2017} shows that even for $\mathbb P^1 \times \mathbb P^1 \times \mathbb P^1$ the converse no longer holds. Similarly, Example 4.9 in \cite{GuVT2008b}  shows that, even in $\mathbb P^1\times\mathbb P^2,$ the converse  is again no longer true. 
%So the converse of Proposition \ref{inclusion-ACM} is not longer true in $\mathbb P^1\times\mathbb P^2$. 
%Then, it is interesting to describe ACM sets of points in $\mathbb P^1\times\mathbb P^n.$ 
The next two examples show how tricky the situation is even in $\mathbb P^1\times\mathbb P^2$. Both of them can be checked by the CoCoA software \cite{cocoa}, but they also follow from Theorem \ref{no intersection}.
\end{remark}

\begin{example}\label{ex:4 pt ACM}
	Let $P_i:=[1,i]\in \mathbb P^1$ and $Q_1,Q_2,Q_3$ three generic points in $\mathbb P^2.$
	Let $$X:=\{P_1\times Q_1, P_2\times Q_2, P_1\times Q_3, P_2\times Q_3 \}.$$
Then $X \subset \mathbb P^1 \times \mathbb P^2$ does not have the $(\star)$-property but it is ACM. (This phenomenon was shown already in \cite{GuVT2008b} Example 4.9, but that example consisted of 27 points while this example uses only four points.)
\end{example}

In \cite{FGM2017} Theorem 3.13 it is shown that, in order to determine the ACM property for a set $X$ of reduced points of $(\mathbb P^1)^n,$ it is enough to show the non-existence of certain sub-configurations of $X$.
The next example proves that a similar characterization of the ACM property is not possible in $\mathbb P^1\times\mathbb P^2$.

\begin{example}\label{ex:6pt Not ACM}
	Let $P_i:=[1,i]\in \mathbb P^1$ and $Q_1,Q_2,Q_3,Q_4,Q_5$  generic points in $\mathbb P^2.$
Then, 
\[
X' : = \{ P_1\times Q_1, P_2\times Q_2, P_1\times Q_3, P_2\times Q_3, P_1\times Q_4, P_2\times Q_4 \} 
\]
is not ACM. However, the set 
\[
X'':=X'\cup\{ P_1\times Q_5, P_2\times Q_5 \}
\]
contains as sub-configuration $X'$ and it is ACM. 
\end{example}

%\begin{proposition}\label{two level sets}
%	Let $P_i:=[1,i]\in \mathbb P^1$ and $\{Q_1,\ldots,Q_N\}$ be a set of $N$ generic points in $\mathbb P^2.$ Let $1\le N'< N''\le N$. 	
%	Then set \[X_1:=\{P_1\times Q_1, \ldots P_1\times Q_{N'}\}\cup\{P_1\times Q_{N''+1}, \ldots P_1\times Q_{N} \}\] and \[X_2:=\{  P_2\times Q_{N'+1}, \ldots  P_2\times Q_{N''} \}\cup\{P_2\times Q_{N''+1}, \ldots P_2\times Q_{N} \}.\] Then $X:=X_1\cup X_2$ is ACM iff %$N-N''\notin \{ {n \choose 2 }-1 | n\in \N     \}.$ 
%\end{proposition}
%\begin{proof}
%	
%\end{proof}
%\textbf{Natural question what about if we add other points?}

The following technical result describes a suitable set of generators for an ACM set of points in $\mathbb P^1\times\mathbb P^n$. 

%{\bf I don't understand the preceding sentence. Isn't this lemma itself the desired description of a suitable generating set?  I have not yet read  carefully this lemma.}

\begin{lemma}\label{gens ACM}
	Let $X$ be an ACM finite set in $\mathbb P^1\times\mathbb P^n.$ Then there exists a set of generators for $I_X$,  $\mathcal G(I_X)\subseteq R$,  such that for any $F\in \mathcal G(I_X)$ we have $F=F'\cdot F''$ where $\deg(F')=(a,0)$  and $\deg(F'')=(0,b)$ for some $a,b\in\N.$ 
\end{lemma}  
\begin{proof}
	Let $X=X_1\cup \cdots\cup X_t$ be the decomposition of $X$ as union of level sets. 
	For $u=1,\ldots,t$, let $H_u\in k[x_0,x_1]$ be the form of degree $(1,0)$ defining the hyperplane containing the points of $X_u$. We introduce for each of these linear forms a new variable, let us call them $z_1, \cdots, z_t$. Let $S$ be the polynomial ring $k[z_1,\ldots z_t,y_0,\ldots, y_n]$. We construct an ideal $J\subseteq S$ by intersecting the prime ideals $(z_u, \ell_{j_1},\ldots, \ell_{j_n} )$ in correspondence to the components of $X$. This intersection defines a height $n+1$ ideal of $S$. 
	
	Consider $J$ as an ideal, say $\overline J$, in the ring $T = S[x_{0}, x_{1}]$, where $S$ is defined in the previous paragraph. Being a cone, $\overline J$ continues to be a height $n+1$ ideal. 
	Consider the linear forms  $z_{u} - H_{u}$,  where $1 \leq i \leq t$. Let $L$ be the ideal generated by all these linear forms. We have
	\[
	R/I_X\cong T/(\overline{J},L), 
	\]
	the former of which is ACM. Since $R/I_X$ and $T/\overline J$  both have height $n+1$, we can view the addition of each linear form in $L$ as a proper hyperplane section, giving that $T/\overline{J}$ is also Cohen-Macaulay.
	
	Note that,  the factorization in the statement  is preserved under proper hyperplane sections, so it if enough to prove the theorem for the ideal $J$. In order to do that we set, for $\mathcal{D}\subseteq [t]:=\{1,2,\ldots, t\}$ $$Y_{\mathcal{D}}:=\bigcup\limits_{i\in[t]\setminus \mathcal{D}} \pi_1( X_i)\subseteq \mathbb P^n.$$
	We denote by $I_{Y_{\mathcal D}}$ the ideal of $S$ generated by the forms in the variables $y_i$'s vanishing in $Y_{\mathcal{D}}$. 
	
	% and  $I_{\mathcal{D}}:=\prod\limits_{j\in \mathcal{D}} H_j\cdot I_{Y_{\mathcal{D}}}$. 
	We also set $$J':=\sum_{\mathcal{D}\subseteq [t] }  I_{Y_{\mathcal D}} \cdot \left(\prod\limits_{j\in D} z_{j}\right) \subseteq S.$$ 
	We claim that $J=J'$ and this will conclude the proof. 
	Note that by construction we have $J\supseteq J'$. To prove the other inclusion, let denote by $D_1, \ldots, D_m$ all the subsets of $[t]$ (the number of level set)  having cardinality $a$, and take	$F\in J$ be a bihomogeneous form of degree $(a,b)$  such that 
	$$ F=\sum_{j=1}^{m}\left( G_j \cdot\prod_{u\in D_j}z_u\right).$$
%	where \begin{enumerate}
%		\item[$i)$] $D_j\subseteq [t]$;
%		\item[$ii)$] $D_j\neq D_{j'}$ for $j\neq j'$;
%		\item[$iii)$] $|D_j|=a$ and $\deg(G_j)=(0,b).$
%	\end{enumerate}  
	Since $X$ is a set of reduced points, $J$ is generated by such forms. %In order to prove that $F\in J'$  we show that, 
	We first show that each summand of $F$ belongs to $J$.  
	Let $k\in\{1,\ldots,m\}$, we set $F_k:=G_k\cdot \prod\limits_{u\in D_k}z_u$ %we need to show that  $G_k \cdot\prod\limits_{u\in D_k}z_u\in I_X,$ so let 
	and take an ideal $\mathfrak p =(z_h, \ell_{j_1},\ldots, \ell_{j_n} )$ in the decomposition of $J$. Two cases occur:
	\begin{itemize}
		\item if $h\in D_k$ then trivially $F_k\in \mathfrak p$;%$G_k \cdot\prod\limits_{u\in D_k}z_u\in (z_h, \ell_{j_1},\ldots, \ell_{j_n} )$;
		\item if $h\notin D_k$, say $P\in \mathbb P^n$ such that $I_P=(\ell_{j_1},\ldots, \ell_{j_n}).$ 
		Then the form $\bar{F}:= \sum\limits_{j\, :\, h\notin D_j} F_j$ vanish at $P.$
%		$$\bar{F}(z_1,\ldots,z_t,y_0,\ldots, y_n):=\sum\limits_{\begin{array}{c}		^{j=1,\ldots,m}\\ ^{j\ \text{s.t.}\  h\notin D_j}\end{array}} \left(G_j\cdot \prod\limits_{u\in D_j}z_u\right)\in I_P,$$
		i.e.  $$\bar{F}(z_1,\ldots,z_t,P)=\sum\limits_{j\, :\, h\notin D_j}\left( G_j(P)\cdot \prod\limits_{u\in D_j}z_u\right) =0.$$
		%Since we have $D_j\neq D_{j'}$ for every $j\neq j'$, we get 
		%$\sum\limits_{\begin{array}{c} 			^{j=1,\ldots,m}\\ ^{j\ \text{s.t.}\  \notin D_j}\end{array}} \left(G_j(P)\cdot \prod\limits_{u\in D_j}z_u\right)=0$ 
		But, since $\prod\limits_{u\in D_j}z_u$ are l.i. in $S_{(a,0)}$, this is true if and only if $G_j(P)=0$. Then in particular $G_k\in I_P$ and $F_k\in \mathfrak p.$
	\end{itemize}
Then $F_k\in J$ and it vanish in each point of $Y_{D_k}$ so $F_k\in I_{Y_{D_k}}\cdot \prod\limits_{j\in D_k} z_{j}$ and we are done.
\end{proof}

\begin{remark}\label{gens2}
	The proof of Lemma \ref{gens ACM} provides a more accurate description of these sets of generators. If $F=F'\cdot F''$  is such an element, then $F'$ is product of linear forms of degree $(1,0),$ each of them defining a hyperplane containing a level set of $X.$ Moreover, if we denote by $X'$  the set of points of $X$ outside the level sets concerning $F'$, then $F''$ is an element in a minimal generating set of $I_{\pi_1(X')}\subseteq k[y_0,\ldots, y_n].$
\end{remark}
	
	Since Proposition \ref{inclusion-ACM} ensures the ACM property for those sets of points with the inclusion property, from now on we focus on sets of points failing the inclusion property.

\begin{notation}
Let $X$ be a set of points in $\mathbb P^1\times\mathbb P^n$ without the inclusion property. We introduce a new partition on $X$.
Let $X:=X_1\cup X_2 \cup \cdots \cup X_t$ be the decomposition of $X$ into level sets.
For any $i=1\ldots,t,$ we set 
\[
Y_i  =  \pi_1(X_i)\subset \mathbb P^n 
\]
and observe that if $P_i \in \eta_1(X) \subset \mathbb P^1$ then
\[
X_i  =   \bigcup \limits_{Q\in Y_i} P_i\times Q.
\]
Then we define $A_X$ and $B_X$ by 
\[
X= A_X\cup B_X 
\]
    %(\textbf{maybe} \ X_A\cup X_B?)
where $P\times Q\in B_X$ if and only if $Q\in \bigcap\limits_{i=1}^{t} Y_i.$  See Figure \ref{def AX and BX}.
We denote $A_i(X):=X_i\cap A_X$ and $B_i(X):=X_i\cap B_X.$
Moreover we set  
\[
Y:=\pi_1(X)\subseteq \mathbb P^n,\ \text{and}\  B_Y:=\pi_1(B_X) \subseteq \mathbb P^n.
\]
%Finally,  for any $Z\subseteq  \mathbb P^n$ we denote by
%$\alpha(Z):=\alpha(I_{Z})$ the initial degree of $I_{Z}\subseteq k[\mathbb P^n]$ i.e. the lowest degree of a non-zero element in $I_{Z}.$

    %
    %(if there is no ambiguity we only write $N_0$, $N_1$ and $\mathcal{D}$.)
\end{notation}

%\begin{remark} 	We will use a property of the binomial coefficients: $${N_0+i+1 \choose n}={N_0+i \choose n}+{N_0+i \choose n-1}.$$\end{remark}

The idea of the above notation is that if we consider the decomposition of $X$ into its $t$ level sets and $\{ P_1,\dots, P_t \} = \eta_1(X)$ then $A_X$ is the set of points $P \times Q \in X$ so that $P_i \times Q \notin X$ for at least one $i$, and $B_X$ is the set of points $P \times Q$ so that $P_i \times Q \in X$ for all $1 \leq i \leq t$.

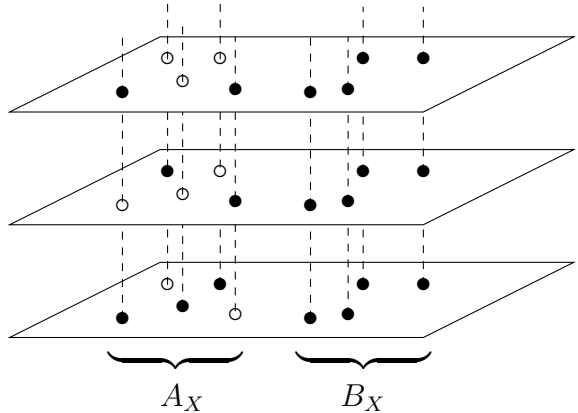
\begin{figure}[h] 

%\begin{center}
\begin{tikzpicture}
\draw [fill=white!10] (8,3) -- (2.5,3) -- (4.5,4) -- (10,4) -- cycle;
\draw [fill=white!10] (8,1.5) -- (2.5,1.5) -- (4.5,2.5) -- (10,2.5) -- cycle;
\draw [fill=white!10] (8,0) -- (2.5,0) -- (4.5,1) -- (10,1) -- cycle;

\node at (4,3.25) {$\bullet$};
\node at (5.5,3.3) {$\bullet$};
\node at (4.6,3.7) {$\circ$};
\node at (5.3,3.7) {$\circ$};
\node at (4.8,3.4) {$\circ$};

\node at (4,1.75) {$\circ$};
\node at (5.5,1.8) {$\bullet$};
\node at (4.6,2.2) {$\bullet$};
\node at (5.3,2.2) {$\circ$};
\node at (4.8,1.9) {$\circ$};

\node at (4,.25) {$\bullet$};
\node at (5.5,.3) {$\circ$};
\node at (4.6,.7) {$\circ$};
\node at (5.3,.7) {$\bullet$};
\node at (4.8,.4) {$\bullet$};

\node at (6.5,.25) {$\bullet$};
\node at (7,.3) {$\bullet$};
\node at (8,.7) {$\bullet$};
\node at (7.2,.7) {$\bullet$};

\node at (6.5,1.75) {$\bullet$};
\node at (7,1.8) {$\bullet$};
\node at (8,2.2) {$\bullet$};
\node at (7.2,2.2) {$\bullet$};

\node at (6.5,3.25) {$\bullet$};
\node at (7,3.3) {$\bullet$};
\node at (8,3.7) {$\bullet$};
\node at (7.2,3.7) {$\bullet$};

\draw[dashed] (4,3.25) -- (4,4);
\draw[dashed] (5.5,3.3) -- (5.5,4);
\draw[dashed] (4.6,3.7) -- (4.6,4.45);
\draw[dashed] (5.3,3.7) -- (5.3,4.45);
\draw[dashed] (4.8,3.4) -- (4.8,4.15);

\draw[dashed] (4,1.8) -- (4,3);
\draw[dashed] (5.5,3) -- (5.5,1.8);
\draw[dashed] (4.6,3) -- (4.6,2.2);
\draw[dashed] (5.3,3) -- (5.3,2.2);
\draw[dashed] (4.8,1.9) -- (4.8,3);

\draw[dashed] (4,.25) -- (4,1.5);
\draw[dashed] (5.5,.35) -- (5.5,1.5);
\draw[dashed] (4.6,.7) -- (4.6,1.5);
\draw[dashed] (5.3,.7) -- (5.3,1.5);
\draw[dashed] (4.8,.4) -- (4.8,1.5);

\draw[dashed] (6.5,.25) -- (6.5,1.5);
\draw[dashed] (7,.3) -- (7,1.5);
\draw[dashed] (8,.7) -- (8,1.5);
\draw[dashed] (7.2,.7) -- (7.2,1.5);

\draw[dashed] (6.5,1.75) -- (6.5,3);
\draw[dashed] (7,1.8) -- (7,3);
\draw[dashed] (8,2.2) -- (8,3);
\draw[dashed] (7.2,2.2) -- (7.2,3);

\draw[dashed] (6.5,3.25) -- (6.5,4);
\draw[dashed] (7,3.3) -- (7,4);
\draw[dashed] (8,3.7) -- (8,4.4);
\draw[dashed] (7.2,3.7) -- (7.2,4.4);

\node at (4.7,-.3) {$\underbrace{\hspace{.7in}}$};
\node at (7.2,-.3) {$\underbrace{\hspace{.7in}}$};
\node at (4.8,-.8) {$A_X$};
\node at (7.2,-.8) {$B_X$};

\end{tikzpicture}

%\end{center}
\caption{Definition of $A_X$ and $B_X$.}
\label{def AX and BX}
\end{figure}

\begin{notation}  
We will need the following invariants.
\[
N_0(X) := \# \pi_{1} (A_X) \ \  \hbox{ and } \ \  N_1(X) := \# \pi_{1} (B_X)
\] 
and 
\[
\mathcal{D}(X):= \bigcup\limits_{i\in \mathbb Z} \left\{ {N_0(X)+i \choose n}, \binom{N_0(X) +i}{n} +1,  \ldots, {N_0(X)+i+1 \choose n}-N_0(X) \right\}.
\]
\end{notation}

%{\bf I changed the range of $i$ from $i \geq 0$ to $i \in \mathbb Z$. For example try $N_0 = 5$, $n=3$. Then $N_1(X) = 4,5$ give ACM but were not in the list with $i \geq 0$.}

\begin{example}
Say $n=2$ and $N_0(X) = 4$. Then $\mathcal D(X)$ is the uncrossed set of numbers in Figure \ref{def of d(x)}. 

\bigskip
\bigskip

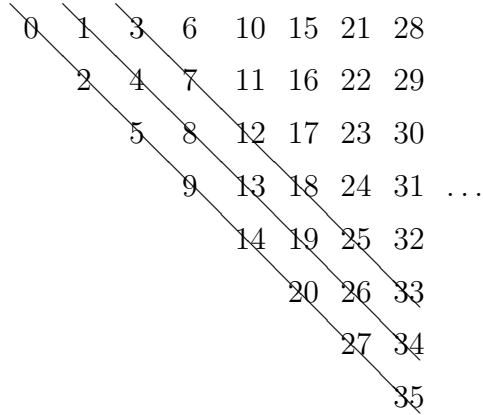
\begin{figure}[h]
\begin{picture}(160,160)(10,10)
\thinlines
\put (10,160){0}
\put (30,160){1}
\put (30,140){2}
\put (50,160){3}
\put (50,140){4}
\put (50,120){5}
\put (70,160){6}
\put (70,140){7}
\put (70,120){8}
\put (70,100){9}
\put (90,160){10}
\put (90,140){11}
\put (90,120){12}
\put (90,100){13}
\put (90,80){14}
\put (110,160){15}
\put (110,140){16}
\put (110,120){17}
\put (110,100){18}
\put (110,80){19}
\put (110,60){20}
\put (130,160){21}
\put (130,140){22}
\put (130,120){23}
\put (130,100){24}
\put (130,80){25}
\put (130,60){26}
\put (130,40){27}
\put (150,160){28}
\put (150,140){29}
\put (150,120){30}
\put (150,100){31}
\put (150,80){32}
\put (150,60){33}
\put (150,40){34}
\put (150,20){35}
\put (170,100){$\dots$}

\put (5,173) {\line(1,-1){155}}
\put (25,173) {\line(1,-1){135}}
\put (45,173) {\line(1,-1){115}}

\end{picture}
\caption{Definition of $\mathcal D(X)$ when $N_0(X) = 4$, $n=2$.}
\label{def of d(x)}
\end{figure}
\noindent By deleting one more diagonal we get the set $\mathcal D(X)$ for the points in Figure 2.
\end{example}

We now describe a construction of a class of sets which are built up by adding points in a certain way, and where we can describe exactly which of the resulting sets are ACM and which are not. To begin we make a stronger assumption on $A_X$, namely that if $P_1 \times Q_1$ and $P_2 \times Q_2$ are both in $A_X$ then $Q_1 \neq Q_2$. That is, $Y_i \cap Y_j \subset B_Y$ for any $i \neq j$. (This is a restriction only if $t \geq 3$.) The result says that if you fix the points of $A_X$ and keep adding generic points to $\pi_1(B_X)$ then $X = A_X \cup B_X$ will switch between being ACM and not being ACM in a predictable way.

%{\bf Somewhere we should say why we need $n \geq 2$ in the following.}
In the next results we will assume the ambient space is $\mathbb P^1\times\mathbb P^n$ where $n \geq 2$.
The exclusion of $n=1$ is not restrictive for this section. Indeed, we are focusing on sets of points failing the inclusion property and, from Proposition \ref{star iff inclusion}  and Theorem 4.11 in \cite{GV-book}, the inclusion property characterize ACM sets of points in $\mathbb P^1\times\mathbb P^1$. 

\begin{theorem}\label{no intersection}
	Let $X\subseteq \mathbb P^1\times\mathbb P^n$, $n \geq 2$, be a finite set without the inclusion property such that the points in $\pi_1(A_X)$ and in $B_Y = \pi_{1}(B_X)$ are generic in $\mathbb P^n$. Moreover, assume $Y_i\cap Y_j\subseteq B_Y$ for any $i \neq j.$ Then  $X$ is ACM if and only if  $N_1(X) \in \mathcal{D}(X)$.
\end{theorem}

\begin{proof}

The coordinate ring for $\mathbb P^1 \times \mathbb P^n$ is $R = k[x_0,x_1,y_0,\dots,y_n]$ with its bihomogeneous grading, which we can also consider with its standard grading as the coordinate ring for $\mathbb P^{n+2}$. From this point of view, a set of points in $\mathbb P^1 \times \mathbb P^n$ can be viewed as a union of lines in $\mathbb P^{n+2}$.

We  make some general observations. Let $X = X_1 \cup \dots \cup X_t$ be the decomposition of $X$ into level sets with respect to $\eta_1$ and let $X' = X_1 \cup \dots \cup X_{t-1}$. If $t=1$ there is only one level set, so the hypothesis that $X$ does not have the inclusion property is impossible. Similarly, if $N_0(X) = 0$ or 1 then $X$ must have the inclusion property. Thus we must assume that $t \geq 2$ and $N_0(X) \geq 2$.

Assume first that $N_1(X) \in \mathcal D(X)$.  We want to show that $X$ is ACM.  Since $N_0(X) \geq N_0(X')$, we get $\mathcal D(X) \subset \mathcal D(X')$ (we remove more diagonals in Figure \ref{def of d(x)} for $\mathcal D(X)$ than for $\mathcal D(X')$).
Thus we  have $N_1(X') = N_1(X)   \in \mathcal D(X) \subseteq \mathcal D(X')$. 

We proceed by induction on $t$. If $t=2$ then it is clear that $X'$ is ACM. Otherwise we can assume that  $X'$ is ACM by induction, since $N_1(X') \in \mathcal D(X')$. We also know that $X_t$ is ACM. We want to show that $X = X' \cup X_t$ is ACM. 

We will view $X, X'$ and $X_t$ as unions of lines in $\mathbb P^{n+2}$.
Consider the exact sequence
\[
0 \rightarrow I_X \rightarrow I_{X'} \oplus I_{X_t} \rightarrow I_{X'} + I_{X_t} \rightarrow 0.
\]
We sheafify this sequence and take cohomology over all twists. Since $X'$ and $X_t$ are ACM unions of lines, we see that $X$ is ACM if and only if $I_{X'} + I_{X_t}$ is a saturated ideal. 

Let $W$ be the scheme in $\mathbb P^{n+2}$ defined by $I_{B_Y} \subset k[y_0,\dots,y_n]$. We note the following facts. 

\begin{itemize}

\item each component of $W$ is defined by $n$ linear forms, hence is a plane.

\item any two components of $W$ meet in a line; in fact, $W$ is the cone over $B_Y$ whose vertex is this line.

\item $W$ is an ACM union of $|B_X|$ planes.

\item $B_i(X) = X_i \cap B_X$ (in $\mathbb P^1 \times \mathbb P^n$) is defined by $I_{B_i(X)} = I_{P_i} + I_{B_X}$, and in $\mathbb P^{n+2}$ is thus a hyperplane section of $W$; its components are lines all passing through a single point.

\item Let $F \in k[x_0,x_1]$ be the product of the linear forms defining the points $\eta_1(X') \subset \mathbb P^1$. Then $B_{X'}$ is defined by the saturated ideal $(F) + I_W$.

\item Let $(P_i,Q_1), (P_j,Q_2) \in B_X$. If $i=j$ and $Q_1 \neq Q_2$ then the corresponding lines in $\mathbb P^{n+2}$ meet in a point; however, this does not affect $X_t \cap X'$ since in this case the two points are either both in $X_t$ or both in $X'$. If $i \neq j$ and $Q_1 \neq Q_2$ then the lines do not meet. If $i \neq j$ and $Q_1 = Q_2$ then the lines meet in the point defined by $I_{P_i} + I_{P_j} + I_{Q_1}$ (which is uniquely determined even if $i$ and $j$ change).

\end{itemize}

The condition $Y_i\cap Y_j\subseteq B_Y$ for any $i \neq j$ implies that the scheme defined by $I_{X'} + I_{X_t}$ has support in the union of points defined by $\bigcap_{Q \in B_Y} I_{P_1} + I_{P_2} + I_Q $. More precisely, if $L_t \in k[x_0,x_1]$ is the linear form defining $P_t$ in $\mathbb P^1$ then the saturation of 
$I_{X'} + I_{X_t}$ is $(L_t,F) + I_W$.

Thus we want to show that if $N_1(X) \in \mathcal{D}(X)$ then $I_{X'} + I_{X_t} = (L_t,F) + I_W$. The inclusion $\subseteq$ is clear, so we must prove $\supseteq$. In particular, 

\begin{quotation}
{\em we have to show that every minimal generator of $I_W$ (which are all in $k[y_0,\dots,y_n]$) is in $I_{X'} + I_{X_t}$.}
 \end{quotation}
 
\noindent If $A_t(X) = \emptyset$ then $I_{X_t} = (L_t,I_W)$ so we are done. Thus in particular we may assume that $A_t(X) \neq \emptyset$.

Now we consider $X_t$. If $|A_t(X)| = N_0(X)$, then the assumption that $Y_i \cap Y_j \subseteq B_Y$ forces all other $A_i(X)$ (if any) to be empty, violating the assumption that the inclusion property does not hold. 
%\textbf{(Does our hypothesis allow this?)} {\bf Why not? It just says that in Figure 2 one of the level sets for $A_X$ has only solid points, not ``open" points, right?} 
%it follows that $\pi_1(X') \subset \pi_1(X)$, so the ACM property for $X$ follows from that of $X'$ by basic double linkage (as was used to prove the inclusion property). 
Hence we can assume that $1 \leq |A_t(X)| \leq N_0(X)-1$. 

Now consider the $h$-vector of $W$. Setting $i$ to be the choice in the definition of $\mathcal D(X)$ that gives $N_1(X)$, we have $\deg W = N_1(X) = \binom{N_0(X)+i}{n} + s$ and the $h$-vector is
\[
\left ( 1, \binom{n}{n-1}, \binom{n+1}{n-1}, \dots, \binom{N_0(X) +i-1}{n-1}, s \right )
\]
where 
\begin{equation}\label{ineq1}
0 \leq s \leq \binom{N_0(X) +i}{n-1} - N_0(X) 
\end{equation}
and the $s$ occurs in degree $N_0(X) + i -n+1$.  The inequality (\ref{ineq1}) means that the number of minimal generators of $I_W$ in degree $N_0(X)+i-n+1$ is at least $N_0(X)$. More importantly, note that the number of minimal generators in degree $N_0(X) +i-n+1$ is exactly $\binom{N_0(X)+i}{n-1} - s$.

We want to see that all of these minimal generators are in $I_{X_t} + I_{X'}$. Lemma \ref{gens ACM} gives a description of the minimal generators of an ACM set of points, but the important thing for us now is to consider the minimal generators that only involve the $y_i$. Let $a$ be the number of points of $A_X$ lying in $X_t$ and $b$ the number not lying on $X_t$. Since the points of $\pi_1(A_X)$ are generic, we have $\binom{N_0(X)+i}{n-1} - s - a$ such minimal generators in $I_{X_t}$ and $\binom{N_0(X)+i}{n-1} - s -b$ such minimal generators in $I_{X'}$. Now we use the assumption that $Y_i\cap Y_j\subseteq B_Y$ and that the points of $\pi_1(A_X)$ are chosen generically. Then the sum has 
\[
2 \cdot \left [ \binom{N_0(X)+i}{n-1} - s \right ] -a -b = 2 \cdot \left [ \binom{N_0(X)+i}{n-1} -s \right ] - N_0
\]
minimal generators involving only the $y_i$ in degree $N_0(X) +i-n+1$. We have to check that this is enough. 
Indeed,
\[
2 \cdot \left [ \binom{N_0(X)+i}{n-1} - s \right ] - N_0 \geq \binom{N_0(X)+i}{n-1} - s
\]
if and only if 
\[
s \leq \binom{N_0(X)+i}{n-1} - N_0(X),
\]
which is equivalent to $N_1(X) \in \mathcal D$. 

Note that $I_W$ may also have minimal generators in degree $N_0(X)+i-n+2$, but this does not interfere with the question of saturation for $I_{X_t} + I_{X'}$.

Note also that this argument simultaneously takes care of the inductive step (taking $t=2$).

The converse is almost the same argument. Indeed, if $X$ is ACM then $I_{X'} + I_{X_n}$ is saturated, and the argument above explains why we must have $N_1(X) \in \mathcal D(X)$.
\end{proof}

\begin{remark} \label{subconfigurations}
As mentioned in the introduction, if $X$ is a finite set of points in $(\mathbb P^{1})^n$ then there is a combinatorial condition on the subsets of $X$ that completely determines whether $X$ is ACM or not.
That is, the ACM question is determined by the existence or not of certain kinds of subconfigurations. Here we see that this is no longer true even in $\mathbb P^1 \times \mathbb P^n$ ($n \geq 2$). Indeed, for sets of points $X = A_X \cup B_X$ satisfying the conditions of the theorem, one can keep adding ``stacked" generic points to $B_X$ (adding one new point in each level set so that $\pi_1 (X)$ only increases by one point),  repeating this procedure as often as desired, and the ACM property will depend only on the cardinality of $B_X$ (since only $N_1(X)$ is increasing, not $N_0(X)$).
\end{remark}

Our next goal is to partially generalize Theorem \ref{no intersection}, removing the assumption $Y_i\cap Y_j\subseteq B_Y$ for any $i\neq j$.

\begin{theorem}\label{yes intersection}
	Let $X \subseteq \mathbb P^1\times\mathbb P^n$ be a finite set without the inclusion property such that the points in $\pi_1(A_X)$ and in $B_Y = \pi_{1}(B_X)$ are generic in $\mathbb P^n$. If $N_1(X) \in \mathcal{D}(X)$ then  $X$ is ACM. %{\bf I have no idea if the converse is true or not, and I don't mind if we don't prove that. But it's fine if we do!}
\end{theorem}

\begin{proof}
We build off Theorem \ref{no intersection}. We know that the result is true when $Y_i\cap Y_j\subseteq B_Y$ for any $i\neq j$, so it is enough to show that adding points one at a time in such a way that $\pi_1(X)$ remains unchanged (equivalently, in this case, such that $\pi_1(A_X)$ remains unchanged) does not affect the ACM property. 

Let $X' \subseteq \mathbb P^1\times\mathbb P^n$ be a finite set such that the points in $\pi_1(A_{X'})$ and in $ \pi_{1}(B_{X'})$ are generic in $\mathbb P^n$, as defined above.  Let $P \in (\mathbb P^1\times\mathbb P^n) \backslash X'$ and for convenience set $P = P_0 \times Q_0$ with $P_0 \in \mathbb P^1$ and $Q_0 \in \mathbb P^n$. Assume that  $\pi_1(P) \in \pi_1(A_{X'})$ and $\pi_2(P) \in  \pi_2(B_{X'})$. (The former says that at least one point of $A_{X'}$ is of the form $P_i \times Q_0$, $i \neq 0$, and the latter says that at least one point of $B_{X'}$ is of the form $P_0 \times Q_i$, $i \neq 0$. In terms of Figure \ref{def AX and BX}, we are allowing ourselves to insert points at the open circles.) Let $X = X' \cup P$. Notice that adding  $P$ to $X'$ in this way gives us $N_0(X) = N_0(X')$ and $N_1(X) = N_1(X')$.

Assume that $N_1(X) = N_1(X') \in \mathcal D(X) = \mathcal D(X')$.  Assume that $X'$ is ACM. We claim that that $X$ is ACM. Then the result will follow from Theorem \ref{no intersection} since we begin with an ACM set of points and keep adding points in a way that preserves the ACM property.

Viewed in $\mathbb P^{n+2}$, we may view $X'$ as a union of lines, so the ACM property is equivalent to the vanishing of $H^1(\mathcal I_{X'}(t))$ for all $t$.  From the long exact sequence associated to the sheafification of the exact sequence
\[
0 \rightarrow I_X \rightarrow I_{X'} \oplus I_P \rightarrow I_{X'} + I_P \rightarrow 0
\]
and the ACM property for $X'$ and for $P$, we see that $X$ is ACM if and only if $I_{X'}+I_P$ is saturated.

Assume without loss of generality that $I_P = (x_0,y_0,\dots,y_{n-1})$.   The ideal $I_{X'} + I_P$ defines the scheme-theoretic intersection of the line associated to $P$ with the union of lines associated to $X'$. This is supported on two points, as follows. 

\begin{itemize}

\item[(a)] If $Q \in X'$ satisfies $\pi_2(Q) = \pi_ 2(P)  $ (i.e. $Q = P_0 \times Q_i$ for some $i \neq 0$), then without loss of generality we can assume that $I_Q = (x_0, \ell_1,\dots, \ell_n)$ where the $\ell_i$ are general linear forms in $k[y_0,\dots,y_n]$. Thus the line in $\mathbb P^{n+2}$ corresponding to $Q$   meets the line corresponding to  $P$ at the point defined by $(x_0,y_0,\dots,y_n)$ in $\mathbb P^{n+2}$, and $I_{X'} + I_P$ defines a scheme supported in part at this point (since by assumption such $Q \in X'$ exist).

\medskip

\item[(b)] If $Q \in X'$ satisfies $\pi_1(Q) = \pi_1(P)$ (i.e. $Q = P_i \times Q_0$ for some $i \neq 0$), without loss of generality assume that $I_Q = (x_1,y_0,\dots, y_{n-1})$. Thus the line in $\mathbb P^{n+2}$ corresponding to  $Q$  meets the line corresponding to  $P$ at the point defined by $(x_0,x_1,y_0,\dots,y_{n-1})$ in $\mathbb P^{n+2}$, and $I_{X'} + I_P$ defines a scheme supported in part at this point (since by assumption such $Q \in X'$ exist). 

\end{itemize}

\noindent No other point of $X'$ corresponds to a line that meets the line corresponding to $P$. Thus the scheme defined by $I_{X'} + I_P$ is supported at these two points.

We now determine $(I_{X'} + I_P)^{sat}$ in $k[x_0,x_1,y_0,\dots,y_n]$. Since it defines a subscheme of a line, clearly it will have the form $(x_0,y_0,\dots,y_{n-1},f)$, where $f \in k[x_1,y_n]$, and we just have to determine $f$. Since we have determined the two points where $f$ vanishes on our line, we also know that $f$ has the form $x_1^\alpha y_n^\beta$ and we only have to determine $\alpha$ and $\beta$.

Let $X_{(a)}$ be the set of points of type (a),  and let $X_{(b)}$ be the set of points of type (b) in $X'$. Let $r$ be the initial degree of $I_{\pi_1(X_{(a)})}$ in $k[y_0,\dots,y_n]$  and let $s = | X_{(b)} |$.

\medskip

\noindent \underline{Claim}: $(I_{X'} + I_P)^{sat} = (x_0,y_0,\dots,y_{n-1}, x_1^s y_n^r)$.

\medskip

Note that $\pi_1(X_{(a)}) \subset \mathbb P^n$ has a homogeneous ideal $J := I_{\pi_1(X_{(a)})} \subset k[y_0,\dots,y_n]$ and is ACM. Hence $I_{X_{(a)}} = (x_0,J)$ in $k[x_0,x_1,y_0,\dots,y_n]$. It follows that 
\[
I_P + I_{X_{(a)}} = (x_0,y_0,\dots,y_{n-1},J) = (x_0,y_0,\dots,y_{n-1},y_n^r)
\]
since the points of $\pi_1(A_{X'})$ and $\pi_1(B_{X'})$ are generic. So the scheme of intersection of the line defined by $P$ with the union of lines corresponding to points of $X_{(a)}$ is defined by the ideal $(x_0,y_0,\dots,y_{n-1},y_n^r)$.

Notice that $\pi_2(X_{(b)}) \subset \mathbb P^1$ is defined by a product of distinct linear forms $G = m_1 \dots m_s$, not divisible by $x_0$, in $k[x_0,x_1]$ and that $X_{(b)}$ is ACM with defining ideal $(G, y_0,\dots,y_{n-1})$. Thus the scheme defined by $I_{X_{(b)}} + I_P$, which is the scheme-theoretic intersection of the line defined by $P$ with the union of lines corresponding to points of $X_{(b)}$ is defined by the ideal $(x_0,x_1^s,y_0,\dots,y_{n-1})$. 

Since $(I_{X'} + I_P)^{sat}$ is the saturated ideal corresponding to the union of these two complete intersection schemes, the claim follows. 

Finally, we have to show that $I_{X'} + I_P = (x_0,y_0,\dots,y_{n-1}, x_1^s y_n^r)$. Let us write $X_{(a)} = A_{(a)} \cup B_{(a)}$, separating the points of $X_{(a)} \cap A_X$ from those of $X_{(a)} \cap B_X$. Let $Y := \pi_1 (X') \backslash \pi_1 (P) \subset \mathbb P^n$. 
The key observation is that the following three are equal:

\begin{itemize}

\item the initial degree of $I_{\pi_1(B_{(a)})}$ in $k[y_0,\dots,y_n]$ (note $B_{(a)} = B_X$);

\item the initial degree of $I_{\pi_1(X_{(a)})}$;

\item the initial degree of $I_Y$. 

\end{itemize}

\noindent This observation is thanks to the numerical assumption $N_1(X') \in \mathcal D(X')$, since $| A_{(a)}| \leq N_0(X') -1$ and the points are generic. Then if $G \in k[x_0,x_1]$ is the generator of $\pi_2(X_{(b)})$ (a product of $s$ distinct linear forms not divisible by $x_0$) and $F$ is a minimal  generator of $I_Y$ of least degree (namely $r$) then clearly $FG \in I_{X'}$ restricts to $x_1^s y_n^r$ modulo $I_P$ and so $I_{X'} + I_P$ is saturated, and hence $X$ is ACM.
\end{proof}

\begin{conjecture}
The converse to Theorem \ref{yes intersection} is also true: if $X = A_X \cup B_X$ is ACM, satisfying the stated assumptions, then $N_1(X) \in \mathcal D(X)$.
\end{conjecture}

\section{Acknowledgements} The first author thanks University of Notre Dame for the hospitality and Elena Guardo for useful discussions. This work was done while the first author was partially supported by the National Group for Algebraic and Geometrical Structures and their Applications (GNSAGA-INdAM) and the second author was partially supported by a Simons Foundation grant  (\#309556).

%%%%%%%%%%%%%%%%%%%%%%%%%%%%%%%%%%%%%%%%%%%%%%%%%%%%%

\end{document}